\documentclass{amsart}
\usepackage{setspace}
\usepackage{a4}
\usepackage{amsthm}
\usepackage{latexsym}
\usepackage{amsfonts}
\usepackage{graphicx}
\usepackage{textcomp}
\usepackage{cite}
\usepackage{enumerate}
\usepackage{amssymb}
\usepackage{hyperref}
\usepackage{amsmath}
\usepackage{tikz}
\usepackage[mathscr]{euscript}
\usepackage{mathtools}
\newtheorem{theorem}{Theorem}[section]

\newtheorem{corollary}[theorem] {Corollary}
\newtheorem{definition}[theorem]{Definition}
\newtheorem{example}[theorem]{Example}

\newtheorem{problem}[theorem]{Problem}

\title{This is the title}
\usepackage{fancyhdr}

\begin{document}
\hrule\hrule\hrule\hrule\hrule
\vspace{0.3cm}	
\begin{center}
{\bf{p-adic Delsarte-Goethals-Seidel-Kabatianskii-Levenshtein-Pfender Bound}}\\
\vspace{0.3cm}
\hrule\hrule\hrule\hrule\hrule
\vspace{0.3cm}
\textbf{K. Mahesh Krishna}\\
School of Mathematics and Natural Sciences\\
Chanakya University Global Campus\\
NH-648, Haraluru Village\\
Devanahalli Taluk, 	Bengaluru  Rural District\\
Karnataka State, 562 110, India \\
Email: kmaheshak@gmail.com

Date: \today
\end{center}

\hrule\hrule
\vspace{0.5cm}
\textbf{Abstract}: We introduce the notion of p-adic spherical codes (in particular, p-adic kissing number problem). We show that the  one-line proof for a variant of the Delsarte-Goethals-Seidel-Kabatianskii-Levenshtein upper bound for spherical codes, obtained by Pfender \textit{[J. Combin. Theory Ser. A, 2007]}, extends to p-adic Hilbert spaces.



\hrule

\hrule
\section{Introduction}
 Let $d\in \mathbb{N}$ and $\theta \in [0, 2\pi)$. A set $\{\tau_j\}_{j=1}^n$  of unit vectors  in $\mathbb{R}^d$	is said to be \textbf{$(d,n,\theta )$-spherical code} \cite{ZONG} in $\mathbb{R}^d$ if 
 	\begin{align}\label{SCI}
 	\langle \tau_j, \tau_k\rangle\leq \cos \theta , \quad \forall 1\leq j, k \leq n, j \neq k.
 \end{align}
Since 
\begin{align*}
	\langle \tau, \omega\rangle =\frac{2-\|\tau-\omega\|^2}{2}, \quad \forall \tau, \omega \in \mathbb{R}^d, 
\end{align*}
we can rewrite Inequality (\ref{SCI}) as 
\begin{align*}\label{SCN}
	\|\tau_j-\tau_k\|\geq  \sqrt{2(1-\cos \theta)}, \quad \forall 1\leq j, k \leq n, j \neq k.	
\end{align*}	
 Fundamental problem associated with spherical codes is the following.
 \begin{problem}\label{SCP}
 	Given $d$ and $\theta$, what is the maximum $n$ such that there exists a $(d,n,\theta )$-spherical code $\{\tau_j\}_{j=1}^n$  in $\mathbb{R}^d$?
 \end{problem}
 The case $\theta=\pi/3$ is known as the famous \textbf{(Newton-Gregory) kissing number problem}. With extensive efforts from many mathematicians, it is still not completely resolved in every dimension (but resolved in dimensions $d=1$ ($n=2$), $d=2$ ($n=6$), $d=3$ ($n=12$), $d=4$ ($n=24$), $d=8$ ($n=240$), $d=24$ ($n=196560$)) \cite{ANSTREICHER, PFENDER, CASSELMAN, MUSIN1, MUSIN2, ODLYZKOSLOANE,  SCHUTTEVANDERWAERDEN,  PFENDERZIEGLER, BACHOCVELLENTIN, MITTELMANNVALLENTIN, BOYVALENKOVDODUNEKOVMUSIN, BOROCZKY, MAEHARA, GLAZYRIN, LIBERTI, LEECH, KALLALKANWANG, JENSSENJOOSPERKINS, MACHADOFILHO, KUKLIN}. We refer  \cite{DELSARTEGOETHALSSEIDEL, BOYVALENKOVDRAGNEVHARDINSTOYANOVA, BARGMUSIN, BOYVALENKOVDRAGNEVHARDINSTOYANOVA2, ERICSONZINOVIEV, SARDARIZARGAR, MUSIN3, BANNAISLOANE, BACHOCVALLENTIN2, CONWAYSLOANE, COHNJIAOKUMARTORQUATO, SAMORODNITSKY, BOYVALENKOVDANEV, BOYVALENKOVDANEVLANDGEV, SLOANE, BANNAIBANNAI, BOROCZKYGLAZYRIN}
  for more on spherical codes.
 Problem \ref{SCP} has connection even with sphere packing  \cite{COHNZHAO}. Most used method for obtaining upper bounds on spherical codes is the Delsarte-Goethals-Seidel-Kabatianskii-Levenshtein
bound which we recall. Let $n \in \mathbb{N}$ be fixed. The Gegenbauer polynomials are defined inductively as 
\begin{align*}
	G_0^{(n)}(r)&\coloneqq1,\quad \forall r \in [-1,1],\\
	G_1^{(n)}(r)&\coloneqq r,\quad \forall r \in [-1,1],\\
	&\quad\vdots\\
	G_k^{(n)}(r)&\coloneqq\frac{(2k+n-4)rG_{k-1}^{(n)}(r)-(k-1)	G_{k-2}^{(n)}(r)}{k+n-3}, \quad \forall r \in [-1,1], \quad \forall k \geq 2.
\end{align*}
Then the family $\{G_k^{(n)}\}_{k=0}^\infty$ is orthogonal on the interval $[-1,1] $ with respect to the weight 
\begin{align*}
	\rho(r)\coloneqq (1-r^2)^\frac{n-3}{2}, \quad \forall r \in [-1,1].
\end{align*}
\begin{theorem} \cite{DELSARTEGOETHALSSEIDEL, ERICSONZINOVIEV} (\textbf{Delsarte-Goethals-Seidel-Kabatianskii-Levenshtein Linear Programming Bound}) \label{DGS}
	Let $\{\tau_j\}_{j=1}^n$  be a  $(d,n,\theta )$-spherical code in $\mathbb{R}^d$. 
	Let $P$ be a real polynomial satisfying following conditions.
	\begin{enumerate}[\upshape(i)]
		\item $P(r)\leq 0$ for all $-1\leq r\leq \cos \theta$.
		\item Coefficients in the  Gegenbauer expansion 
		\begin{align*}
			P=\sum_{k=0}^{m}a_kG_k^{(n)}
		\end{align*}
	satisfy 
	\begin{align*}
a_0>0, \quad a_k\geq 0, ~\forall 1\leq k \leq m.
	\end{align*}
	\end{enumerate}
Then
\begin{align*}
	n \leq \frac{P(1)}{a_0}.
\end{align*}
\end{theorem}
 In 2007, Pfender gave a one-line proof for a variant of Theorem  \ref{DGS}.
 \begin{theorem}  \cite{PFENDER} (\textbf{Delsarte-Goethals-Seidel-Kabatianskii-Levenshtein-Pfender Bound}) \label{PFENDERT}
 	Let $\{\tau_j\}_{j=1}^n$  be a  $(d,n,\theta )$-spherical code in $\mathbb{R}^d$. Let $c>0$ and $\phi:[-1,1]\to \mathbb{R}$ be a function satisfying following.
 	\begin{enumerate}[\upshape(i)]
 		\item 
 		\begin{align*}
 			\sum_{j=1}^n\sum_{k =1}^n\phi(\langle \tau_j, \tau_k\rangle)\geq 0.	
 		\end{align*}
 		\item $\phi(r)+c\leq 0$ for all $-1\leq r \leq \cos \theta$.
 	\end{enumerate}
 Then 
 \begin{align*}
 	n\leq \frac{\phi(1)+c}{c}.
 \end{align*}
In particular, if $\phi(1)+c\leq 1$, then $n\leq 1/c$.
 \end{theorem}
In this paper, we introduce the notion of p-adic spherical codes. We show that Theorem  \ref{PFENDERT},  can be easily extended for p-adic Hilbert spaces.

 \section{p-adic spherical Codes}
We begin from  the   definition  of  p-adic Hilbert space.  
\begin{definition}\cite{DIAGANABOOK, KALISCH, DIAGANARAMAROSON} \label{PADICDEF}
	Let $\mathbb{K}$ be a non-Archimedean complete valued field (with valuation $|\cdot|$) and $\mathcal{X}$ be a non-Archimedean Banach space (with norm $\|\cdot\|$) over $\mathbb{K}$. We say that $\mathcal{X}$ is a \textbf{p-adic Hilbert space} if there is a map (called as p-adic inner product) $\langle \cdot, \cdot \rangle: \mathcal{X} \times \mathcal{X} \to \mathbb{K}$ satisfying following.
	\begin{enumerate}[\upshape (i)]
		\item If $x \in \mathcal{X}$ is such that $\langle x,y \rangle =0$ for all $y \in \mathcal{X}$, then $x=0$.
		\item $\langle x, y \rangle =\langle y, x \rangle$ for all $x,y \in \mathcal{X}$.
		\item $\langle \alpha x+y, z \rangle =\alpha \langle x,  z \rangle+\langle y,z\rangle$ for all  $\alpha  \in \mathbb{K}$, for all $x,y,z \in \mathcal{X}$.
		\item $|\langle x, y \rangle |\leq \|x\|\|y\|$ for all $x,y \in \mathcal{X}$.
	\end{enumerate}
\end{definition}
Following is the  standard example which we consider in the paper.
\begin{example}\cite{KALISCH}
	Let $p$ be a prime. For $d \in \mathbb{N}$, let $\mathbb{Q}_p^d$ be the standard p-adic Hilbert space equipped with the inner product 
	\begin{align*}
		\langle (a_j)_{j=1}^d,(b_j)_{j=1}^d\rangle := \sum_{j=1}^da_jb_j,  \quad \forall (a_j)_{j=1}^d,(b_j)_{j=1}^d \in \mathbb{Q}_p^d
	\end{align*}
	and norm 
	\begin{align*}
		\|(x_j)_{j=1}^d\|:= \max_{1\leq j \leq d}|x_j|, \quad \forall (x_j)_{j=1}^d\in 	\mathbb{Q}_p^d.
	\end{align*}
\end{example}
We introduce p-adic spherical codes  as follows. 
\begin{definition}\label{PCODEDEFINITION}
	Let $d\in \mathbb{N}$ and $\theta \in [0, 2 \pi)$. A set $\{\tau_j\}_{j=1}^n$  of vectors  in $ \mathbb{Q}_p^d$	is said to be \textbf{p-adic $(d,n,\theta)$-spherical code}  in $\mathbb{Q}_p^d$ if following conditions hold.
	\begin{enumerate}[\upshape(i)]
		\item $	\|\tau_j\|=1$ for all  $1\leq j \leq n$.
		\item $	\langle \tau_j, \tau_j\rangle=1$ for all  $1\leq j \leq n$.
		\item  
		\begin{equation}\label{PE}
			|2-2\langle \tau_j, \tau_k\rangle|\geq 2(1-\cos \theta), \quad \forall 1\leq j, k \leq n, j \neq k.
		\end{equation}
		\end{enumerate}
		We call the  case $\theta=\pi/3$  as the \textbf{p-adic kissing number problem}.
\end{definition}
Let $\{\tau_j\}_{j=1}^n$  be a   p-adic $(d,n,\theta )$-spherical code  in $\mathbb{Q}_p^d$. 
Since 
\begin{align*}
	\|\tau_j-\tau_k\|^2\geq |\langle \tau_j-\tau_k,  \tau_j-\tau_k\rangle|=|2-2\langle \tau_j, \tau_k\rangle|, \quad  \forall 1\leq j, k \leq n, 
\end{align*}
Inequality   (\ref{PE}) gives
\begin{align}\label{PN}
	\|\tau_j-\tau_k\|\geq \sqrt{2(1-\cos \theta)},	\quad  \forall 1\leq j, k \leq n,  j \neq k.
\end{align}
However, note that Inequality (\ref{PN}) may not give Inequality   (\ref{PE}). Note that we can formulate the definition of  general p-adic $(d,n,\theta)$-spherical code by replacing Inequality (\ref{PE})  with Inequality (\ref{PN}) in Definition \ref{PCODEDEFINITION}. It is also possible to consider Definition \ref{PCODEDEFINITION} by removing conditions (i) or (ii).
Following is the p-adic version of Theorem \ref{PFENDERT}.
\begin{theorem} (\textbf{p-adic Delsarte-Goethals-Seidel-Kabatianskii-Levenshtein-Pfender Spherical Codes Bound}) \label{PDGSP}
	Let $\{\tau_j\}_{j=1}^n$  be a  p-adic $(d,n,\theta )$-spherical code in $\mathbb{Q}_p^d$. Let $c>0$ and $\phi:[0, \infty)\to \mathbb{R}$ be a function satisfying following.
	\begin{enumerate}[\upshape(i)]
		\item $\sum_{1\leq j,k \leq n}\phi(|2-2\langle \tau_j, \tau_k\rangle|)\geq 0$.
		\item $\phi(r)+c\leq 0$ for all $ r \in [2(1-\cos \theta), \infty)$.
	\end{enumerate}
	Then 
	\begin{align*}
		n\leq \frac{\phi(0)+c}{c}.
	\end{align*}
	In particular, if $\phi(0)+c\leq 1$, then $n\leq 1/c$.	
\end{theorem}
\begin{proof}
	Define $\psi:[0, \infty)\ni r \mapsto \psi(r)\coloneqq \phi(r)+c\in \mathbb{R}$. Then 
	\begin{align*}
		\sum_{1\leq j,k \leq n}\psi(|2-2\langle \tau_j, \tau_k\rangle|)&=\sum_{j=1}^{n}\psi(|2-2\langle \tau_j, \tau_j\rangle|)+\sum_{1\leq j,k \leq n, j \neq k}\psi(|2-2\langle \tau_j, \tau_k\rangle|)\\
		&=\sum_{j=1}^{n}\psi(0)+\sum_{1\leq j,k \leq n, j \neq k}\psi(|2-2\langle \tau_j, \tau_k\rangle|)\\
		&=n(\phi(0)+c)+\sum_{1\leq j,k \leq n, j \neq k}(\phi(|2-2\langle \tau_j, \tau_k\rangle|+c)\\
		&\leq n(\phi(0)+c)+0=n(\phi(0)+c).
	\end{align*}
	We also have 
	\begin{align*}
		\sum_{1\leq j,k \leq n}\psi(\phi(|2-2\langle \tau_j, \tau_k\rangle|))=\sum_{1\leq j,k \leq n}(\phi(|2-2\langle \tau_j, \tau_k\rangle|)+c)=\sum_{1\leq j,k \leq n}\phi(|2-2\langle \tau_j, \tau_k\rangle|)+cn^2.
	\end{align*}
	Therefore 
	\begin{align*}
		cn^2\leq \sum_{1\leq j,k \leq n}\phi(|2-2\langle \tau_j, \tau_k\rangle|)+cn^2\leq n(\phi(0)+c).
	\end{align*}
\end{proof}
\begin{corollary} (\textbf{p-adic Delsarte-Goethals-Seidel-Kabatianskii-Levenshtein-Pfender Kissing Number Bound})
	Let $\{\tau_j\}_{j=1}^n$  be a  p-adic $(d,n,\pi/3)$-spherical code in $\mathbb{Q}_p^d$. Let $c>0$ and $\phi:[0, \infty)\to \mathbb{R}$ be a function satisfying following.
	\begin{enumerate}[\upshape(i)]
		\item $\sum_{1\leq j,k \leq n}\phi(|2-2\langle \tau_j, \tau_k\rangle|)\geq 0$.
		\item $\phi(r)+c\leq 0$ for all $ r \in [1, \infty)$.
	\end{enumerate}
	Then 
	\begin{align*}
		n\leq \frac{\phi(0)+c}{c}.
	\end{align*}
	In particular, if $\phi(0)+c\leq 1$, then $n\leq 1/c$.	
\end{corollary}	
Following generalization of Theorem \ref{PDGSP} is easy.
\begin{theorem}
	Let $\{\tau_j\}_{j=1}^n$  be a  p-adic $(d,n,\theta )$-spherical code in $\mathbb{Q}_p^d$. Let $c>0$ and 
	\begin{align*}
		\phi:\{|2-2\langle \tau_j, \tau_k\rangle|:1\leq j, k \leq n\}\to \mathbb{R}	
	\end{align*}
	be a function satisfying following.
	\begin{enumerate}[\upshape(i)]
		\item $\sum_{1\leq j,k \leq n}\phi(|2-2\langle \tau_j, \tau_k\rangle|)\geq 0$.
		\item $\phi(r)+c\leq 0$ for all $ r \in \{|2-2\langle \tau_j, \tau_k\rangle|:1\leq j, k \leq n, j \neq k\}$.
	\end{enumerate}
	Then 
	\begin{align*}
		n\leq \frac{\phi(0)+c}{c}.
	\end{align*}
	In particular, if $\phi(0)+c\leq 1$, then $n\leq 1/c$.	
\end{theorem}
Note that, in the paper, we can replace $\mathbb{Q}_p^d$ by any p-adic Hilbert space.

 \bibliographystyle{plain}
 \bibliography{reference.bib}

\end{document}